\documentclass[12pt, oneside]{article}

\usepackage{tikz}
\usetikzlibrary{arrows}
\tikzstyle{block}=[draw opacity=0.7,line width=1cm]

\usepackage{graphicx}
\usepackage{amsmath, amsthm, amsfonts}
\usepackage[authoryear]{natbib}
\usepackage{algorithm, algorithmic}
\usepackage{placeins}
\usepackage{rotating}
\usepackage{booktabs}
\usepackage{multirow} 
\usepackage[font=small,labelfont=bf,tableposition=top]{caption}
\usepackage{xfrac}
\usepackage{relsize}
\usepackage{mathrsfs}
\usepackage{lineno}
\allowdisplaybreaks		
\usepackage{tikz}
\usepackage{tkz-tab}
\usepackage{latexsym}
\usepackage{epstopdf} 					
\usepackage{pgfplots} 					

\usepackage{afterpage}

\usepackage[text={16cm,23cm},centering]{geometry}
\newcommand{\RN}[1]{%
	\textup{\uppercase\expandafter{\romannumeral#1}}%
}			
\newcommand{\rn}[1]{%
	\textup{\lowercase\expandafter{\romannumeral#1}}%
}			

\usepackage{subcaption}
\usepackage{lscape}
\usepackage{longtable}
\usepackage{verbatim}
\usepackage[colorlinks,citecolor=blue,urlcolor=blue,bookmarks=blue,hypertexnames=blue]{hyperref}
\usepackage{enumitem}
\setlist[itemize]{}
\usepackage{authblk}

\newtheorem{theorem}{Theorem}

\newtheorem{proposition}{Proposition}

\title{An Economies-of-Scale Service System Design Problem}

\author{Pooya Hoseinpour}
\affil{\footnotesize Department of Industrial Engineering \& Management Systems, Amirkabir University of Technology (Tehran Polytechnic), Iran\\
	\footnotesize \text{p.hoseinpour@aut.ac.ir} \\ }
\date{}

\begin{document}
	
	\maketitle
	\begin{abstract}
		This paper studies the design of a service system in the presence of economies-of-scale. The goal is to make decision on the number, location, service capacities of facilities as well as on the allocation of customers to the opened facilities in order to minimize the cost of whole system. The total cost includes opening and serving costs of facilities aggregated to the transportation and waiting costs of customers. To reflect the economies-of-scale in the modeling, a general opening cost function is supposed for each service facility with the characteristic of being concave and non-decreasing on its service capacity. A Lagrangian relaxation algorithm is developed for solving the problem in its general form. The algorithm decomposes the relaxed model into some homogeneous subproblems where each one can be optimally solved in polynomial time with no need for any optimization solver. Our computational experiment shows that the developed algorithm is both efficient and effective in solving the proposed problem.	
	\end{abstract}
	
	\newpage
	\section{Introduction}\label{Int}
	A service system is primarily consisting of two key components, namely, service facilities and customers, where the customers create demands for getting the service offered by service facilities. Both of the demand levels and the service times are stochastic by nature which inevitably results in congestion at service facilities. Congestion in a service facility can be seen as a key factor that has a severe impact on service quality. A service system design (SSD) problem tries to balance the service providing cost and service offering quality. This is one of the interesting problems in location science and also known as stochastic facility location problems with congestion. For a comprehensive review of this problem, the interested reader is referred to the surveys \cite{boffey2007review} and \cite{berman2015stochastic}.
	
	In a service system, the service facilities can be either immobile or mobile; in the first case, the customers are supposed to travel to the locations of service facilities for getting service; however, the service facilities travel to the customers' locations in the second case. This work explores designing an immobile service system, where some potential locations are predetermined for the service facilities. One needs to make decision on both of the numbers and the location of service facilities. Moreover, the service capacity of each opened service facility is often assumed to be optimally decided. There are two approaches for incorporating this decision in the modeling; some works simply assume a finite set of capacity levels for service provision \citep[see e.g.,][]{elhedhli2006service,vidyarthi2014efficient,aboolian2012profit,ahmadi2017convexification,ahmadi2018service, hoseinpour2019designing}, the others model it as a continuous real-valued non-negative decision variable \citep[see e.g.,][]{wang2004facility,castillo2009social,hoseinpour2016profit,elhedhli2018service, ahmadi2018location}. We apply the second approach here in our modeling and will find a closed-form solution for the service capacity of each service facility in a general setting, which would directly be a function of the customers' demands who are allocated to the service facility. 
	
	The allocation of customers to the opened service facilities are assumed to be either user-choice or direct-choice. In a user-choice allocation, each customer selects a service facility for getting service, with the goal of maximizing her utility \citep[see survay][]{berman2015stochastic}; however, in a direct-choice allocation, the customers behave in a way that the allocations be optimal for the whole system, called socially optimal solution. In other words, the goal in the socially-optimal decision making, to which our model belongs, is to minimize the cost of the whole system which is an aggregated function of the cost terms originated from both components, i.e., service facilities and customers. From the facilities' side, typically, facility opening and service costs are considered; whereas, accessing and congestion costs are added from the customers' side. One can see that the first one shows the system's owner cost and the second one reflects the service quality issue. See e.g., \cite{wang2002algorithms,wang2004facility,elhedhli2006service,aboolian2008location,berman2007multiple,kim2013column,vidyarthi2014efficient,hoseinpour2016profit,ahmadi2017convexification,ahmadi2018service,elhedhli2018service, ahmadi2018location, hoseinpour2019designing} for recent works who apply this approach. There is also another approach for considering the service quality, where a threshold is taken for some measure of quality, such as the average number of customers waiting in the queue. See for example the recent works \cite{rajagopalan2001capacity,silva2008locating,baron2008facility,aboolian2012profit}.  
	
	In order to precisely decide the optimal number and service capacity of service facilities, it is highly suggested to consider the advantages of economies-of-scale in the modeling. This has been extensively considered in the classical facility location models \citep[see e.g.,][]{dasci2001plant,dupont2008branch,baumgartner2012supply,saif2016lagrangian}, but is rare in SSD problems. There is strong evidence in the reality for the existence of economies-of-scale in service industries; for example in the banking industry \citep{doukas1991economies}, and healthcare \citep{preyra2006scale}. 
	
	The benefits of economies-of-scale have been previously considered in the design of a service system, where discrete service capacity levels are assumed for each service facility. Technically, considering economies-of-scale would be much easier if some service capacity levels are predefined since only some modifications in the parameters are needed to meet this assumption. Whereas, it makes the structure of the proposed models to be too complicated when the service capacity of each service facility is a continuous variable. Because of this, all previous works who assumed continuous service capacity ignored the benefits of considering economies-of-scale and just modeled the opening cost at each service facility as a linear function of service capacity. The only exception is \cite{elhedhli2018service} who defines $h(\mu):=\sqrt{\mu}$ as an economies-of-scale opening cost of the service facility that works with capacity $\mu$. Our work extends their work and considers a general $h(\mu)$ function, with characteristics of being concave and non-decreasing. The modeling approach is completely novel which enables us to easily handle the complexity of considering economies-of-scale service opening costs. For doing so, the opening cost function is first approximated with its piece-wise linearization and then a novel solution algorithm based on Lagrangian relation is constructed where the relaxed model can be decomposed to smaller subproblems. A polynomial-time exact algorithm is developed for optimally solving the relaxed submodels. Using the proposed algorithm for solving the submodels, the optimal solution of the model can be found directly with no need for any optimization solver, such as CPLEX.   
	
	The remainder of the paper is organized as follows. Section \ref{ProMod} mathematically states the problem which can be solved using a novel Lagrangian relaxation algorithm, which is developed in Section \ref{SolPro}. Section \ref{ComRes} reports the computational results of considering different opening functions. Finally, Section \ref{Con} concludes the main findings with some directions for future research. 
	
	\section{Problem Modeling}\label{ProMod}
	This section mathematically models the problem, i.e., the economies-of-scale SSD problem. The goal is to find the optimal value of decision variables to minimize the whole cost of the service system. Assume $I$ shows the set of potential service facility locations and $J$ be the set of all customers. Let $h_i(\mu_{i})$, the economies-of-scale opening cost function for $i \in I$, is defined as
	\begin{equation}\label{eq:1}
		h_i(\mu_{i}):=f_{i}+c_{i}g_i(\mu_{i}),
	\end{equation}
	where $f_{i}$ is the fixed cost of opening service facility $i \in I$ per time unit, $c_{i}$ is the operating cost in service facility $i \in I$ per time unit, and $g_i(.)$ is an arbitrary concave non-decreasing function. Let $s_{i}$ denotes the service cost per time unit at service facility $i \in I$ and $w_{i}$ shows the waiting cost of each customer in service facility $i \in I$ per time unit. $a_{ij}$ shows the access cost of customer $j\in J$ per time unit if she is assigned to the service facility $i\in I$. Finally, $\lambda_{j}$ denotes the demand rate of customer $j\in J$. The goal is to find the optimal value of the following decision variables:
	
	\begin{itemize}
	 \item $\mu_{i}$ is a non-negative decision variable shows the service rate in service facility $i \in I$,
	 \item $x_{i}\in \{0,1\}$ is a binary decision variable that is one if service facility at location $i \in I$ is opened to provide service,
	 \item $y_{ij}\in \{0,1\}$ is one if customer $j\in J$ is assigned to service facility $i\in I$.
	\end{itemize}
	
	The queue of each opened service facility $i\in I$ is assumed to be an M/M/1 queue system and the arrival rate of customers is shown by $\Lambda_i$, which is equal to 
	\begin{equation}\label{eq:2}
		\Lambda_i=\sum_{j}{\lambda_{j}y_{ij}} \qquad i\in I,
	\end{equation}
	and serving each customer, independently, takes a random time that follows from an exponential distribution with the expected value of $1/\mu_i$, and the discipline of the queue is FCFS, first come first severed. Although we limit this work to M/M/1 queue in each service facility, which is widely used in the literature, the proposed methodology can easily be applied to SSD problems with other queue models in the service facilities. 
	 
	\begin{proposition}\emph{\citep{shortle2018fundamentals}}\label{Pro:1}
	The steady-state condition in the M/M/1 queue system of service facility $i\in I$ is $\rho_{i}<1$, where $\rho_{i}:=\Lambda_{i}/\mu_{i}$ and $\Lambda_{i}$ is defined by Eq.\eqref{eq:2}.
	\end{proposition}	
	\begin{proposition}\emph{\citep{shortle2018fundamentals}}\label{Pro:2}
		The average waiting time of each customer in the M/M/1 queue system of service facility $i\in I$ is $W_{i}^{q}=\frac{1}{\mu_i-\Lambda_i}$, where $\Lambda_{i}$ is defined by Eq.\eqref{eq:2}.
	\end{proposition}

	The mathematical program for our discussed problem, the economies-of-scale SSD, can be proposed as follows:
	\begin{subequations}\label{EoS-SSD Model}
		\begin{eqnarray}
			\text{(EoS-SSD)} 	&& \nonumber\\			
			\underset{x_i,\mu_i,y_{ij}}{\text{minimize }}
			&&\sum_{i\in I}\left\lbrace 
			h_i(\mu_{i})x_{i} + \sum_{j\in J}s_{i}\lambda_{j}y_{ij} + \sum_{j\in J}a_{ij}\lambda_{j}y_{ij} 
			+ w_{i}\Lambda_{i}W_{i}^{q}\right\rbrace \label{EoS-SSD Model:1}\\						
			\text{subject to}
			&& {y_{ij}} \leq x_{i} \qquad i\in I,\; j\in J\label{EoS-SSD Model:2}\\
			&& \sum_{i\in I} y_{ij}=1 \qquad j\in J \label{EoS-SSD Model:3}\\
			&& \mu_{i}-\Lambda_{i} \geq 0 \qquad i\in I \label{EoS-SSD Model:4}\\
			&& \mu_{i} \geq 0 \qquad i\in I \label{EoS-SSD Model:5}\\
			&& x_{i} \in \{0,1\} \qquad i\in I \label{EoS-SSD Model:6}\\
			&& y_{ij}\in \{0,1\} \qquad i\in I,\; j\in J, \label{EoS-SSD Model:7}						 
		\end{eqnarray}
	\end{subequations}
	where the first term in the objective function \eqref{EoS-SSD Model:1} is the total opening cost of service facilities in the service system and the function $h_{i}(.)$ is defined by Eq.\eqref{eq:1}. The second term calculates the total serving cost of customers, the third term is customers accessing cost, and the last term is customer waiting cost. In the last term of the objective function, the total value of customers waiting cost can be calculated multiplying $\Lambda_{i}$, from Eq.\eqref{eq:2}, by average waiting cost of each customer calculated in Proposition \ref{Pro:2}. One can say that the first two terms are facility-side cost, and the last two terms reflect customer' service quality; by aggregating them the model seeks the socially optimal solution. Constraint set \eqref{EoS-SSD Model:2} indicates that there is no assignment for the non-opened service facility. Constraint set \eqref{EoS-SSD Model:3} ensures that each customer has been assigned to exactly one service facility for placing her demand. Constraint set \eqref{EoS-SSD Model:4} is the steady-state condition obtained in Proposition \ref{Pro:1}. Constraint set \eqref{EoS-SSD Model:5} shows that the service rate of each service facility should be a non-negative real value. Finally, Constraint sets \eqref{EoS-SSD Model:6} and \eqref{EoS-SSD Model:7} show binary location-allocation decision variables.   
	
	The $\text{EoS-SSD}$ model, i.e., Model \eqref{EoS-SSD Model}, is too complex to be solved for an arbitrary opening cost function. In order to deal with this, we first approximate $g_i(\mu_i)$, for $i \in I$, by its piece-wise linear function $\hat{g}_i(\mu_i)$, which can be defined as
	\begin{equation}\label{eq:4}
		\hat{g}_i(\mu_i):= \underset{k\in K}{\text{min}}\left\lbrace \hat{g}_i^{k}(\mu_i)\right\rbrace,
	\end{equation}
	where $\hat{g}_i^{k}(\mu_i)$ is the linear function that is tangent to $g_i(\mu_i)$ in point $\mu^k$ and can be defined as $$\hat{g}_i^{k}(\mu_i):= g_i(\mu^{k})+g_i^{\prime}(\mu^{k})(\mu_i-\mu^k),$$ for $k \in K$. See Figure \ref{Capacity} where function $\hat{g}_i(\mu_i)$ is approximated by $\underset{k\in K}{\text{min}}\left\lbrace \hat{g}_i^{k}(\mu_i)\right\rbrace$.
	
	\begin{figure}[t]
		\begin{center}
			\includegraphics[height=2.8in]{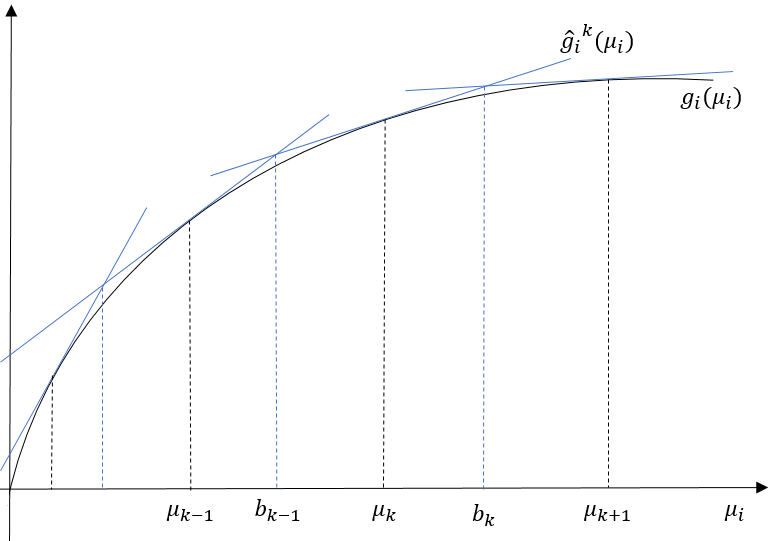}
			\caption{piece-wise linearization of function $g_i(\mu_i)$, $i \in I$, which is a general concave non-decreasing function} \label{Capacity}
		\end{center}
	\end{figure}
	
	{  
		\begin{tabular}{llp{12.5cm}}
			\tabularnewline
			\hline 
			\multicolumn{3}{l}{} \bf{Algorithm 1.} \hfill \tabularnewline
			\hline 
			&\hspace{4mm} \bf{Step 1.} 		& Set $k\leftarrow1$, and $b_{0}=\underline{\upsilon}$. \tabularnewline
			&\hspace{4mm} \bf{Step 2.}  	& Knowing $b_{k-1}$, find $\mu_k$ by solving following equation $$(1+\epsilon)g_i(b_{k-1})=g_i(\mu_k)+g_i^{\prime}(\mu_k)(b_{k-1}-\mu_k),$$ where breakpoint $b_{k-1}$ satisfies in line $\hat{g}_i^{k}(\mu_i)$ and $\hat{g}_i^{k}(b_{k-1})=(1+\epsilon)g_i(b_{k-1})$.\tabularnewline
			&\hspace{4mm} \bf{Step 3.}  	& Knowing $\mu_k$, find $b_k$ by solving following equation $$(1+\epsilon)g_i(b_k)=g_i(\mu_k)+g_i^{\prime}(\mu_k)(b_{k}-\mu_k),$$ where breakpoint $b_{k}$ satisfies in line $\hat{g}_i^{k}(\mu_i)$ and $\hat{g}_i^{k}(b_{k})=(1+\epsilon)g_i(b_k)$.  \tabularnewline
			&\hspace{4mm} \bf{Step 4.}  	& If $b_{k}\leq \bar{\upsilon}$ set $k\leftarrow k+1$ and go Step 2.   \tabularnewline		
			\hline 
		\end{tabular}
	}
	\newline
	
	One can find the finite set of breakpoints $b_k$, and tangent points $\mu_k$, for $k\in K$, such that $0\leq \frac{\hat{g}_i(\mu_i)-g_i(\mu_i)}{g_i(\mu_i)} \leq \epsilon$, for an arbitrary error threshold of $\epsilon >0$. For this purpose, Algorithm 1 is developed inspiring from \cite{elhedhli2005exact}. In Algorithm 1, Steps 2 and 3 iteratively find tangent points and breakpoints, respectively. Moreover, $\underline{\upsilon}$ and $\bar{\upsilon}$ are determined numerically considering the type of function $g_i(.)$.
	
	Substituting \eqref{eq:4} in the objective function \eqref{EoS-SSD Model:1} results 
	\begin{equation}\label{eq:13}
		\underset{x_i,\mu_i,y_{ij}}{\text{minimize }}
		\sum_{i\in I}\underset{k\in K}{\text{min}}\left\lbrace v^{i,k}(x_i,\mu_i,y_{ij})\right\rbrace,
	\end{equation}
	where
	\begin{align}\label{eq:5}
		v^{i,k}(x_i,\mu_i,y_{ij}):=&\left(f_i+c_ig_i(\mu^{k})-c_i\mu^kg_i^{\prime}(\mu^{k})\right)x_{i}+c_ig_i^{\prime}(\mu^{k})\mu_i \nonumber\\
		+&\sum_{j\in J}s_{i}\lambda_{j}y_{ij} + \sum_{j\in J}a_{ij}\lambda_{j}y_{ij} 
		+ w_{i}\frac{\sum_{j\in J}\lambda_{j}y_{ij}}{\mu_i-\sum_{j\in J}\lambda_{j}y_{ij}},
	\end{align}
	for $i\in I$, and $k\in K$. In next section, we will develop a novel Lagrangian relaxation algorithm to solve Model \eqref{eq:13}, \eqref{EoS-SSD Model:2}-\eqref{EoS-SSD Model:7} efficiently.  
	
	\bigbreak
	\section{The Lagrangian Relaxation Algorithm}\label{SolPro}
	In this section, a Lagrangian relaxation algorithm is developed to efficiently solve the proposed Model \eqref{eq:13}, \eqref{EoS-SSD Model:2}-\eqref{EoS-SSD Model:7}. An interested reader is referred to, e.g., \cite{fisher2004lagrangian,geoffrion2010lagrangian} for more details on Lagrangian relaxation method. In abstract, in a Lagrangian relaxation algorithm, first the hard constraints, by relaxing which the model will be much easier to solve, are determined and dualized by Lagrangian multipliers. For $\text{EoS-SSD}$ model, the hard constraints \eqref{EoS-SSD Model:3} are relaxed with Lagrangian multipliers $u_{j}, j\in J$ and aggregated to the objective function, \eqref{EoS-SSD Model:1}, which yields the relaxed Model $\text{R-EoS-SSD}$ as follows:
	\begin{subequations}\label{R-EoS-SSD}
		\begin{eqnarray}
			\left(\text{R-EoS-SSD}\right) && \nonumber\\
			\underset{x_i,\mu_i,y_{ij}}{\text{minimize}} 
			&&\sum_{i\in I}\left\lbrace\underset{k\in K}{\text{min}}\left\lbrace v^{i,k}(x_i,\mu_i,y_{ij})\right\rbrace \right\rbrace
			+ \sum_{j\in J}u_j\left(1-\sum_{i\in I}y_{ij}\right)  \label{R-EoS-SSD:1}\\		
			\text{subject to}&& \eqref{EoS-SSD Model:2}, \eqref{EoS-SSD Model:4}-\eqref{EoS-SSD Model:7} \nonumber,						 
		\end{eqnarray}
	\end{subequations}
	where $v^{i,k}(x_i,\mu_i,y_{ij})$ is defined by \eqref{eq:5}. 
	
	For any values of Lagrangian multipliers $u_j$, for $j \in J$, the optimal value of relaxed problem $\text{R-EoS-SSD}$ provides a lower bound for the optimal value of the original problem $\text{EoS-SSD}$. The goal is to find the tightest Lagrangian lower bound. For doing so, the Lagrangian dual problem is defined as
	\begin{subequations}\label{LDual}
		\begin{eqnarray}
			\left(\text{D-EoS-SSD}\right) && \nonumber\\
			\underset{\bf{u}}{\text{maximum }}       &&L(\bf{u})	\\
			&& \bf{u}\in \mathbb{R}^{|J|},					 
		\end{eqnarray}
	\end{subequations}
	where $L(\bf{u})$ is the optimal objective function value of $\text{R-EoS-SSD}$, i.e., \eqref{R-EoS-SSD:1}. 
	
	This problem is normally solved using the well-known subgradient algorithm, which is an iterative algorithm. Let $\mathbf{u}^t$ be the values of Lagrangian multipliers in iteration $t$. The subgradient algorithm starts by setting initial values for the Lagrangian multipliers, shown as $\mathbf{u}^0$, and updates them iteratively. In each iteration $t$, the relaxed problem is optimally solved and provides a lower bound solution $Lb^t$. As the solution might not be feasible to the original problem, a heuristic algorithm is then used to construct a feasible solution, which will be an upper bound solution for the original problem, shown as $Ub^t$. The Lagrangian multipliers for next iteration will be updated as follow:
	\begin{equation}\label{LMulti}
		\mathbf{u}^{t+1}\leftarrow \mathbf{u}^{t}+\frac{\alpha^t\left(\underset{s=0,\dots,t}{\text{min}}\left\lbrace Ub^s\right\rbrace -Lb^t\right) }{||\bf{v}^t||}\mathbf{v}^t , 
	\end{equation}
	where $\mathbf{v}_j^t=1-\sum_{i\in I}{y_{ij}}$ for $j \in J$. $\alpha^t$ is a control parameter starts with initial $\alpha^t\in (0,2)$ and is reducing when there is no improvements for long time; see \cite{geoffrion2010lagrangian}. The algorithm stops when the lower bound and upper bound solutions are close enough to each other, i.e., 
	\begin{equation}\label{LError}
		\frac{\underset{s=0,\dots,t}{\text{min}}\left\lbrace Ub^s\right\rbrace -Lb^t}{Lb^t}\leq e,
	\end{equation}
	where $e>0$ is a user-defined error bound. 
	
	In the following, Subsection \ref{LowBou} provides a methodology to polynomially solve the R-EoS-SSD in each iteration of Lagrangian relaxation algorithm. Subsection \ref{UppBou} presents a heuristic algorithm that constructs a feasible solution from the solution obtained by solving R-EoS-SSD.  
	
	\bigbreak
	\subsection{Lower bound Solution}\label{LowBou}
	Let set the values of Lagrangian multipliers as $\mathbf{u}^t$ in iteration $t$. The relaxed Model \eqref{R-EoS-SSD} is separable to $|I|$ subproblem, called $\text{R-EoS-SSD}^{(i),t}$, as follows: 
	\begin{subequations}\label{R-EoS-SSD-i}
		\begin{eqnarray}
			\left(\text{R-EoS-SSD}^{(i),t}\right) && \nonumber\\
			\underset{\mu_i,y_{ij}}{\text{minimize}}
			&&\underset{k\in K}{\text{min}}\left\lbrace v^{i,k}(1,\mu_i,y_{ij})-\sum_{j\in J}u^t_{j}y_{ij}\right\rbrace \label{R-EoS-SSD-i:1}\\
			\text{subject to}&& \mu_{i}-\sum_{j\in J}\lambda_{j}y_{ij} \geq 0 \label{R-EoS-SSD-i:3}\\
			&& \mu_{i} \geq 0 \label{R-EoS-SSD-i:4}\\
			&& y_{ij}\in \{0,1\} \qquad j\in J. \label{R-EoS-SSD-i:5}						 
		\end{eqnarray}
	\end{subequations}
	
	After solving $\text{R-EoS-SSD}^{(i),t}$, for $i \in I$, if the objective value is negative, then set $x_{i}=1$ and open the service facility $i \in I$; otherwise set $x_{i}=0$ and close the service facility with objective value of zero. Theorem \ref{Th:1} finds the optimal solution of problem $\text{R-EoS-SSD}^{(i),t}$, for $i \in I$, in polynomial time. See Proposition \ref{Pr:3} for more detail.
	
	\begin{theorem}\label{Th:1}
	The optimal solution of problem $\text{R-EoS-SSD}^{(i),t}$, for $i \in I$ in iteration $t$, is 
	\begin{align*}
		y^*_{ij}&= argmin\left\lbrace z^{i,k(i)} \right\rbrace\\
		\mu^*_{i}&=\sum_{j\in J}{\lambda_{j}y^*_{ij}}+\sqrt{\frac{w_i\sum_{j\in J}{\lambda_{j}y^*_{ij}}}{c_ig_i^{\prime}(\mu^{k(i)})}},
	\end{align*}
	where $k(i)=\underset{k\in K}{argmin}\left\lbrace z^{i,k}\right\rbrace$ and
	\begin{subequations}\label{Inner}
		\begin{eqnarray}
		z^{i,k}-p^{i,k}=\underset{y_{j}}{\text{minimize }}
		&&\sum_{j\in J}\left(q_j^{i,t,k}-u_j^t\right)y_{j} + \sqrt{\sum_{j\in J}r_j^{i,k}y_{j}}\\
		\text{subject to}	&& y_{j}\in \{0,1\} \qquad j\in J,					 
		\end{eqnarray}
	\end{subequations}
	for $i \in I$ and $k \in K$, where
	\begin{align*}
		p^{i,k}&=f_{i}+c_{i}g_i(\mu^k)-c_{i}g_i^{\prime}(\mu^k)\mu^k\\
		r_j^{i,k}&=4w_{i}c_{i}g_i^{\prime}(\mu^k)\lambda_j\\
		q_j^{i,k}&=s_{i}\lambda_{j}+a_{ij}\lambda_j+c_{i}g_i^{\prime}(\mu^k)\lambda_j.
	\end{align*}
	\end{theorem}	
 
	\begin{proof}{Proof of Theorem \ref{Th:1}.}
		By changing the order of minimizations in problem  $\text{R-EoS-SSD}^{(i),t}$, i.e., \eqref{R-EoS-SSD-i}, it's optimal solution can be found by solving $|K|$ similar problems as follows:
		\begin{subequations}\label{R-SSD-i-k}
			\begin{eqnarray}
				\underset{\mu_i,y_{ij}}{\text{minimize}} 
				&&v^{i,k}(x_i,\mu_i,y_{ij})-\sum_{j\in J}u^t_{j}y_{ij} \label{R-SSD-i-k:1}\\
				\text{subject to}	&& \eqref{R-EoS-SSD-i:3}-\eqref{R-EoS-SSD-i:5} \nonumber,						 
			\end{eqnarray}
		\end{subequations}
		and choosing the optimal solution of the one with minimum objective value. In order to solve \eqref{R-SSD-i-k} one can first find the optimal value of $\mu_{i}$ by considering the following problem: 
		\begin{subequations}\label{Cap}
			\begin{eqnarray}
				\underset{\mu_i}{\text{minimize }}
				&&c_ig_i^{\prime}(\mu^k)\mu_i+w_i\frac{\Lambda_i}{\mu_i-\Lambda_i}\\
				\text{subject to}&& \mu_i \geq \Lambda_i \label{Cap:1}.						 
			\end{eqnarray}
		\end{subequations}
		The program \eqref{Cap} is convex respect to $\mu_i$; therefore, by setting its differentiation equal zero, one can find the optimal value of service capacity as follows
		\begin{equation}\label{OptCap}
			\mu_i^{*}=\Lambda_i+\sqrt{\frac{w_i\Lambda_i}{cg_i^{\prime}(\mu^{k})}}.
		\end{equation}
		which clearly satisfies Constraint \eqref{Cap:1}. Substituting the optimal value of service capacity \eqref{OptCap} into objective function \eqref{R-SSD-i-k:1}	which completes the proof.
	\end{proof}
	
	Model \eqref{Inner}, defined in Theorem \ref{Th:1}, is a pure integer program and one can optimally solve it using the proposed Algorithm 2 in the following.
	
	{
		\begin{tabular}{llp{12.5cm}}
			\tabularnewline
			\hline 
			\multicolumn{3}{l}{} \bf{Algorithm 2.} \hfill \tabularnewline
			\hline 
			&\hspace{4mm} \bf{Step 1.} 		& For each $j\in J$, if $\left(q_j^{i,t,k}-u_j^t\right)\geq 0$, put $j$ in set $J_1$ and if $\left(q_j^{i,t,k}-u_j^t\right)<0$ and $r_j^{i,k}=0$, put it in set $J_2$; otherwise put it in set $J_3$.     \tabularnewline
			&\hspace{4mm} \bf{Step 2.}  	& For customer $j\in J_{1}$, set $y_{ij}=0$ and for customer $j\in J_{2}$ set $y_ij=1$. \tabularnewline
			&\hspace{4mm} \bf{Step 3.}  	& Sort customers $j\in J_{3}$ in non-decreasing order of $\left\lbrace \frac{q_j^{i,t,k}-u_j^t}{r_j^{i,k}}\right\rbrace$. \tabularnewline
			&\hspace{4mm} \bf{Step 4.}  	& Define $J_3^{(m)} \subseteq J_3$ such that contain first $m$ customers in $J_3$, and calculate $$Q_m:=\sum_{j\in J_3^{(m)}} \left(q_j^{i,t,k}-u_j^t\right) + \sqrt{\sum_{j\in J_3^{(m)}} r_j^{i,k}} $$ for $m \in J_3$. Find $m^*:=\underset{m \in J_3}{\text{argmin }}{Q_m}$.  \tabularnewline
			&\hspace{4mm} \bf{Step 5.}  	& For $j \in J_3$ that $j\leq m^*$ set $y_{ij}=1$; for $j \in J_3$ that $j>m^*$ set $y_{ij}=0$.	\tabularnewline	
			\hline 
	\end{tabular}}
	\newline 
	
	Algorithm 2 is driven directly from Theorem \ref{Th:2}.	
	\begin{theorem}\emph{\citep{daskin2002inventory}} \label{Th:2}
		Consider problem
		\begin{align*}
			\text{minimize}\quad &h(\mathbf{y})=\sum_{j\in J}d_{j}y_{j}+\sqrt{\sum_{j\in J}b_{j}y_{j}}\\
			\text{subject to\quad}	 & y_{j}\in \{0,1\} \qquad j\in J
		\end{align*}
		with $d_{j}\in \mathbb{R},\; b_{j} \in \mathbb{R}^+$ for $j\in J$. 
		
		For each $j \in J$, if $d_{j}\geq 0$, then $y_{j}=0$; if $d_{j}<0$ and $b_{j}=0$, then $y_{j}=1$; otherwise for ordered $j$ as
		\begin{equation*}
			\left\lbrace \frac{d_{1}}{b_{1}} \right\rbrace \leq \dots \leq  \left\lbrace \frac{d_{|\hat{J}|}}{b_{|\hat{J}|}} \right\rbrace,
		\end{equation*}
		in set $\hat{J}\subseteq J$, if $y_{s}=1$, then $y_{l}=1$ for all $l \leq s$.
	\end{theorem}
	
	\begin{proposition}\label{Pr:3}
		The complexity of Algorithm 2 is $\text{O}\left(|I|\times|K|\times|J|\log{|J|}\right)$.
	\end{proposition}
	\begin{proof} {Proof of Proposition \ref{Pr:3}.}
		The complexity of algorithm for finding minimum of $n$ numbers is $\text{O}(n)$ and the complexity of an efficient sorting algorithm of $n$ numbers is $\text{O}(n\log n)$.  
	\end{proof}
	
	Based on Proposition \eqref{Pr:3}, one can conclude that Theorem \eqref{Th:1} finds the optimal solution of problem $\text{R-EoS-SSD}^{(i),t}$ in low-order polynomial time.
	
	\bigbreak
	\subsection{Upper bound Solution}\label{UppBou}
	The solution provided by the relaxed problem, i.e, $\text{R-EoS-SSD}$, might be infeasible to the original problem. For example, some customers may be assigned to more than one service facility and some customers may not be assigned to any opened service facility. To construct a feasible solution from the solution obtained from the relaxed problem, one can use the following simple heuristic algorithm:
	
	{
		\begin{tabular}{llp{12.5cm}}
			\tabularnewline
			\hline 
			\multicolumn{3}{l}{} \bf{Algorithm 3.} \hfill \tabularnewline
			\hline 
			&\hspace{4mm} \bf{Step 1.} 		& Open service facilities same as the values of $x_i$, for $i \in I$, obtained by solving the relaxed model $\text{R-EoS-SSD}$.    \tabularnewline
			&\hspace{4mm} \bf{Step 2.}  	& Do following substeps:\tabularnewline
			&\hspace{4mm} \it{ Step 2.1.}  	& For each customer $j\in J$ who is assigned to at least one service facility, based on the values of $y_{ij}$ obtained from solving the relaxed model $\text{R-EoS-SSD}$, select service facility $i$ such that $y_{ij}=1$ and has minimum value of $\left\lbrace \frac{q_j^{i,k}}{r_j^{i,k}}\right\rbrace $.    \tabularnewline
			&\hspace{4mm} \it{ Step 2.2.}  	& For each customer $j\in J$ who is assigned to no service facility, based on the values of $y_{ij}$ obtained from solving the relaxed model $\text{R-EoS-SSD}$, select service facility $i$ such that $x_i=1$ and has minimum value of $\left\lbrace \frac{q_j^{i,k}}{r_j^{i,k}}\right\rbrace $.   \tabularnewline
			&\hspace{4mm} \bf{Step 3.}  	& Close an open service facility if no customer is assigned to it in Step 2.  \tabularnewline		
			\hline 
	\end{tabular}}
	\newline 
	
	In Step 1, as the same service facilities as the ones opened in lower bound solution with the same determined capacity are selected to be opened. Step 2 assigns the customers to the opened service facilities. For this purpose, first, for the customers who are assigned to at least one service facility based on the lower bound solution, the service facility with less increase in the objective function is determined in Step 2.1. For the customers who have already no assignments, the service facility among all opened service facilities is selected which increases in the objective function less. Finally in Step 3 if there is an opened service facility with no assignments, it will be closed.       
	
	\bigbreak
	\section{Numerical Testing}\label{ComRes}
	The proposed Lagrangian relaxation algorithm in Section \ref{SolPro} has been implemented in C++ to solve the SSD problem with no need for any commercial solver. The tests are done on a machine with an Intel Core i5 1.8GHz CPU and 8GB RAM. The same test problems as in \cite{elhedhli2006service} are picked and modified. The test problems are initially proposed by \cite{holmberg1999exact} for a class of facility location problems and later some of them are later modified and used by \cite{elhedhli2006service} for designing an SSD problem. In our problem, the service capacities of facilities are decision variables. Thus, the test problems whose parameters are similar to others, with the exception of their facility capacities, are excluded. 27 test problems are finally selected for the current experiment where the number of service facilities and customers varies from 10 to 30 and 50 to 150, respectively. The parameters of the test problems are modified as follows. The serving cost of a customer in service facilities, $s_i$, $i \in I$, are randomly generated from the interval [1 , 5]. The waiting cost of each customer per day, $w_i$, $i \in I$, is randomly generated from the interval [50 , 300]. The fixed cost of opening service facility, $f_i$, $i \in I$, the demand rates of customers, i.e., $\lambda_{j}$, $j \in J$, and the accessing cost of customers to their service facilities, i.e., $a_{ij}$, $i \in I, j \in J$, are exactly collected from the original test problems.
	
	The proposed solution method enables us to consider any arbitrary function with the general setting of being concave and non-decreasing for defining the opening cost of service facilities, i.e., $h_{i}(\mu_{i})$, $i \in I$. Here, three different functions have been defined for opening cost of service facilities; namely, linear, square-root, and fractional functions. The first one, the linear function is defined as
	\begin{equation*}
		h_{i}(\mu_i)=f_i+c_i\mu_i \qquad i \in I,
	\end{equation*}	
	which was traditionally used for mathematically showing opening cost of service facilities; see for examples \cite{castillo2009social, hoseinpour2016profit, ahmadi2018location}. The second one, the square-root function, is defined as
	\begin{equation*}
		h_{i}(\mu_i)=f_i+c_i\sqrt{\mu_i} \qquad i \in I.
	\end{equation*}
	
	This function has been recently considered by \cite{elhedhli2018service}. However, their proposed approach for modeling can only handle square function. The third one is defined here as a concave fractional function as follows:
	\begin{equation*}
		h_{i}(\mu_i)=f_i+\frac{c_i\mu_i}{\mu_i + 1} \qquad i \in I.
	\end{equation*}	
	 	
 	Here, we fix the operating cost of each service facility, i.e., $c_i$ as 1, 10, and 100 for linear, square-root, and fractional opening costs, respectively. 

	Using Algorithm 1 when $g(\mu_i)=\sqrt{\mu_i}$, one can use $$\mu_k\leftarrow\left((1+\epsilon)\sqrt{b_{k-1}}+\sqrt{\left(\epsilon^2+2\epsilon\right)b_{k-1}}\right)^2$$ to update $\mu_k$ in Step 2 and use $$b_k\leftarrow\left((1+\epsilon)\sqrt{\mu_k}+\sqrt{\left(\epsilon^2+2\epsilon\right)\mu_k}\right)^2$$ to find $b_k$ in Step 3. When $g(\mu_i)=\frac{\mu_i}{\mu_i+1}$, one can use $$\mu_k\leftarrow\frac{(1+b)\sqrt{\epsilon b_{k-1}}+(1+\epsilon)b_{k-1}}{1-\epsilon b_{k-1}} $$ to update $\mu_k$ in Step 2 and use $$b_k\leftarrow\frac{1}{2}\left(2\mu_k+\epsilon(1+\mu_k)^2+(1+\mu_k)\sqrt{\epsilon^2(1+\mu_k)^2+4\epsilon\mu_k}\right)$$  to find $b_k$ in Step 3. 
	
	All instances are solved using the developed Lagrangian relaxation algorithm with user-defined tolerance of $e=0.01$ and maximum 10000 iterations of the running algorithm. The algorithm will stop when at least one of these criteria be satisfied. Tables \ref{Linear}-\ref{Fractional} respectively presents the computational results of solving the generated instances considering the three above-mentioned functions for facility opening cost, namely, linear, square-root, and fractional functions. To have a fair comparison, we fixed $\alpha^0$, the initial value of the control parameter in Eq. \ref{LMulti} equal 0.01 and it is halved if there is no significant improvement in lower bounds for 10 consecutive iterations.
	
	The columns of each table respectively show the instance name, picked from \cite{holmberg1999exact}, number of potential service facilities, number of customers, total cost of designing the service system, percentage of each cost component from the total cost, i.e., the opening, serving, accessing, and waiting costs, number of the iterations the algorithm runs for reaching the user-defined tolerance or the predefined limit for maximum number of iterations, the tolerance when the algorithm stops running, the CPU time of running the algorithm, number of opened service facilities, and the average service capacity in the opened facilities. As shown, the proposed algorithm is very successful in solving the instances, where most instances reach the defined tolerance in less than 24 seconds. It is also seen that more iterations are required for the algorithm to reach the user-defined tolerance in the case that the economies-of-scale opening functions are assumed. Moreover, the CPU time is almost higher for square-root and fractional functions. This happens since some breakdown points have been calculated for piece-wise linearization of the opening function and the algorithm spends much longer time for finding the lower bound solution in each iteration of the solution algorithm.
	
	By switching from linear opening cost to the ones benefits from the economies-of-scale, i.e., square-root or fractional opening cost, the total cost slightly decreases, see Figure \ref{EoSTC}. Although the total cost might not change significantly in the presence of economies of scale, the system opens fewer service facilities with higher service capacities on average. This causes the waiting costs to decrease in the whole system, which has a higher intensity in fractional function rather than the square-root one; see Figure \ref{EoSWC} for better comparison per instance.

\begin{figure}[t]
	\begin{center}
		\includegraphics[height=2.8in]{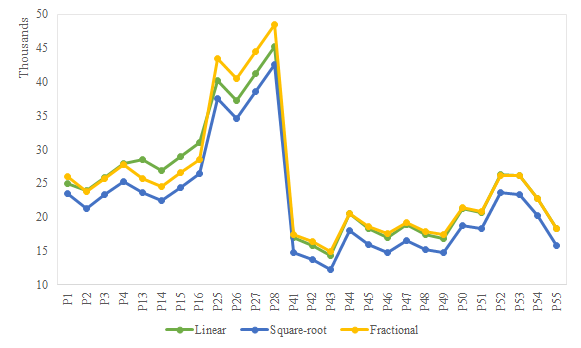}
		\caption{Comparison on total cost considering different opening cost function.} \label{EoSTC}
	\end{center}
\end{figure}	

\begin{figure}[t]
	\begin{center}
		\includegraphics[height=2.8in]{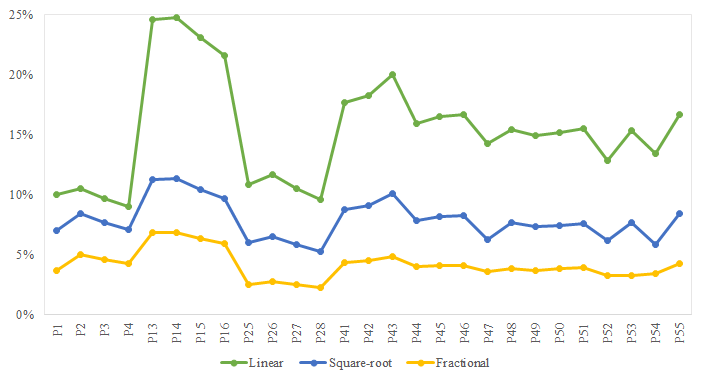}
		\caption{Comparison on waiting cost percentage considering different opening cost function.} \label{EoSWC}
	\end{center}
\end{figure}

	\bigbreak
	\section{Conclusions}\label{Con}
	In this paper, we study an economies-of-scale design of a service system where the opening cost of service facilities are assumed to be any arbitrary function, with characteristics of being concave and non-decreasing. The goal is to optimally decide the number, location, and service capacity of opened facilities as well as the allocation of customers to the facilities. The problem is modeled as a mixed-integer nonlinear program that is too complex to be solved for any kind of opening function. Therefore, it is first approximated by its piece-wise linear functions and then a Lagrangian decomposition algorithm is developed, in which the relaxed model is decomposed into smaller subproblems. A polynomial-time exact algorithm is developed for solving each submodel where the optimal service capacity of each facility is drawn as a closed-form expression. The proposed solution algorithm is successful in solving the instances of the problem; however, considering economies-of-scale increases the solving time in comparison to the traditional model in which a simple linear opening function is considered. Our computational results show that in the presence of economies-of-scale, fewer service facilities are opened with higher service capacity. This decreases the average waiting time of customers and subsequently increases the service quality.
	
	One can consider other kinds of queuing systems in service facilities, e.g., queue with general service time such as M/G/1, as a direction for future research. Moreover, one can consider a profit-maximizing service system where the facilities can use the benefits of economies-of-scale to serve the customers, serving who are not traditionally profitable for the system.
	
\afterpage{
	\begin{landscape}
		\begin{table}
			\centering
			\caption{Computational results when facility opening cost is a linear function of service capacity}
			\resizebox{\linewidth}{!}{
				\begin{tabular}{lllllllllllll}
					\hline 
					Instance & \# Service Facilities & \# Customers & Total Cost & Opening Cost & Serving Cost & Accessing Cost & Waiting Cost & \# Iterations & CPU time(ms) & Error Tolerance & \# Open Facilities & Average Service Capacity\tabularnewline
					\hline 
					P1 & 10 & 50 & 25027.4 & 23\% & 19\% & 48\% & 10\% & 9 & 0 & 0.005 & 1 & 318\tabularnewline
					P2 & 10 & 50 & 23877.4 & 19\% & 20\% & 50\% & 10\% & 11 & 0 & 0.002 & 1 & 318\tabularnewline
					P3 & 10 & 50 & 25877.4 & 25\% & 19\% & 46\% & 10\% & 8 & 0 & 0.007 & 1 & 318\tabularnewline
					P4 & 10 & 50 & 27877.4 & 31\% & 17\% & 43\% & 9\% & 4 & 0 & 0.009 & 1 & 318\tabularnewline
					P13 & 20 & 50 & 28450.6 & 25\% & 17\% & 34\% & 25\% & 11 & 0 & 0.003 & 1 & 790\tabularnewline
					P14 & 20 & 50 & 26958.0 & 25\% & 18\% & 32\% & 25\% & 8 & 0 & 0.010 & 1 & 755\tabularnewline
					P15 & 20 & 50 & 28958.0 & 30\% & 16\% & 30\% & 23\% & 7 & 0 & 0.011 & 1 & 755\tabularnewline
					P16 & 20 & 50 & 30958.0 & 35\% & 15\% & 28\% & 22\% & 5 & 0 & 0.010 & 1 & 755\tabularnewline
					P25 & 30 & 150 & 40259.5 & 30\% & 32\% & 27\% & 11\% & 89 & 125 & 0.010 & 2 & 309\tabularnewline
					P26 & 30 & 150 & 37264.3 & 24\% & 34\% & 30\% & 12\% & 85 & 110 & 0.010 & 2 & 308\tabularnewline
					P27 & 30 & 150 & 41264.3 & 32\% & 31\% & 27\% & 11\% & 80 & 109 & 0.009 & 2 & 308\tabularnewline
					P28 & 30 & 150 & 45264.3 & 38\% & 28\% & 24\% & 10\% & 77 & 172 & 0.008 & 2 & 308\tabularnewline
					P41 & 10 & 90 & 17009.8 & 33\% & 13\% & 37\% & 18\% & 39 & 16 & 0.010 & 1 & 500\tabularnewline
					P42 & 20 & 80 & 15807.8 & 34\% & 12\% & 35\% & 18\% & 5 & 0 & 0.001 & 1 & 473\tabularnewline
					P43 & 30 & 70 & 14311.3 & 37\% & 12\% & 30\% & 20\% & 7 & 0 & 0.009 & 1 & 452\tabularnewline
					P44 & 10 & 90 & 20486.1 & 28\% & 13\% & 43\% & 16\% & 35 & 0 & 0.010 & 1 & 550\tabularnewline
					P45 & 20 & 80 & 18250.5 & 31\% & 13\% & 40\% & 17\% & 27 & 15 & 0.010 & 1 & 520\tabularnewline
					P46 & 30 & 70 & 16990.9 & 32\% & 12\% & 39\% & 17\% & 17 & 16 & 0.010 & 1 & 489\tabularnewline
					P47 & 10 & 90 & 18871.2 & 31\% & 14\% & 41\% & 14\% & 55 & 16 & 0.010 & 1 & 549\tabularnewline
					P48 & 20 & 80 & 17403.5 & 33\% & 14\% & 38\% & 15\% & 53 & 16 & 0.010 & 1 & 525\tabularnewline
					P49 & 30 & 70 & 16791.7 & 33\% & 12\% & 40\% & 15\% & 60 & 32 & 0.010 & 1 & 490\tabularnewline
					P50 & 10 & 100 & 21212.7 & 27\% & 11\% & 46\% & 15\% & 8 & 0 & 0.007 & 1 & 530\tabularnewline
					P51 & 10 & 100 & 20695.9 & 28\% & 12\% & 45\% & 16\% & 64 & 32 & 0.010 & 1 & 530\tabularnewline
					P52 & 10 & 100 & 26298.9 & 23\% & 11\% & 54\% & 13\% & 16 & 0 & 0.008 & 1 & 578\tabularnewline
					P53 & 20 & 100 & 26177.2 & 37\% & 14\% & 34\% & 15\% & 295 & 156 & 0.010 & 3 & 239\tabularnewline
					P54 & 10 & 100 & 22792.3 & 26\% & 13\% & 47\% & 13\% & 59 & 16 & 0.010 & 1 & 592\tabularnewline
					P55 & 20 & 100 & 18356.9 & 33\% & 16\% & 34\% & 17\% & 39 & 15 & 0.010 & 1 & 592\tabularnewline
					\hline 
				\end{tabular}
			}
			\label{Linear}
		\end{table}
	\end{landscape}
}	

\afterpage{
	\begin{landscape}
		\begin{table}
			\centering
			\caption{Computational results when facility opening cost is a square-root function of service capacity}
			\resizebox{\linewidth}{!}{
				\begin{tabular}{lllllllllllll}
					\hline 
					Instance & \# Service Facilities & \# Customers & Total Cost & Opening Cost & Serving Cost & Accessing Cost & Waiting Cost & \# Iterations & CPU time(ms) & Error Tolerance & \# Open Facilities & Average Service Capacity\tabularnewline
					\hline 
					P1 & 10 & 50 & 23421.1 & 19\% & 21\% & 53\% & 7\% & 13 & 312 & 0.010 & 1 & 1631\tabularnewline
					P2 & 10 & 50 & 21305.2 & 21\% & 11\% & 59\% & 8\% & 16 & 265 & 0.010 & 1 & 2282\tabularnewline
					P3 & 10 & 50 & 23305.2 & 28\% & 10\% & 54\% & 8\% & 8 & 250 & 0.002 & 1 & 2282\tabularnewline
					P4 & 10 & 50 & 25305.2 & 34\% & 10\% & 50\% & 7\% & 5 & 250 & 0.002 & 1 & 2282\tabularnewline
					P13 & 20 & 50 & 23642.3 & 21\% & 20\% & 48\% & 11\% & 17 & 1110 & 0.010 & 1 & 3053\tabularnewline
					P14 & 20 & 50 & 22399.7 & 21\% & 21\% & 46\% & 11\% & 12 & 1062 & 0.009 & 1 & 2915\tabularnewline
					P15 & 20 & 50 & 24399.7 & 28\% & 20\% & 43\% & 10\% & 8 & 1047 & 0.001 & 1 & 2915\tabularnewline
					P16 & 20 & 50 & 26399.7 & 33\% & 18\% & 39\% & 10\% & 6 & 1062 & 0.002 & 1 & 2915\tabularnewline
					P25 & 30 & 150 & 37604.0 & 29\% & 34\% & 31\% & 6\% & 107 & 23125 & 0.010 & 3 & 735\tabularnewline
					P26 & 30 & 150 & 34624.0 & 23\% & 37\% & 34\% & 6\% & 117 & 23500 & 0.010 & 2 & 886\tabularnewline
					P27 & 30 & 150 & 38624.0 & 31\% & 33\% & 30\% & 6\% & 94 & 22812 & 0.010 & 2 & 886\tabularnewline
					P28 & 30 & 150 & 42624.0 & 37\% & 30\% & 27\% & 5\% & 84 & 21579 & 0.010 & 2 & 886\tabularnewline
					P41 & 10 & 90 & 14756.0 & 30\% & 15\% & 47\% & 9\% & 7 & 813 & 0.001 & 1 & 1720\tabularnewline
					P42 & 20 & 80 & 13692.6 & 32\% & 14\% & 45\% & 9\% & 69 & 2922 & 0.010 & 1 & 1627\tabularnewline
					P43 & 30 & 70 & 12275.7 & 35\% & 15\% & 40\% & 10\% & 82 & 4484 & 0.010 & 1 & 1558\tabularnewline
					P44 & 10 & 90 & 17958.6 & 25\% & 15\% & 53\% & 8\% & 8 & 813 & 0.002 & 1 & 1890\tabularnewline
					P45 & 20 & 80 & 15920.2 & 28\% & 15\% & 49\% & 8\% & 54 & 2797 & 0.010 & 1 & 1786\tabularnewline
					P46 & 30 & 70 & 14834.5 & 29\% & 14\% & 48\% & 8\% & 75 & 4453 & 0.010 & 1 & 1684\tabularnewline
					P47 & 10 & 90 & 16588.4 & 27\% & 16\% & 51\% & 6\% & 29 & 890 & 0.010 & 1 & 2130\tabularnewline
					P48 & 20 & 80 & 15173.1 & 29\% & 16\% & 48\% & 8\% & 7 & 2547 & 0.006 & 1 & 1804\tabularnewline
					P49 & 30 & 70 & 14749.8 & 30\% & 14\% & 49\% & 7\% & 81 & 4469 & 0.010 & 1 & 1686\tabularnewline
					P50 & 10 & 100 & 18774.0 & 24\% & 13\% & 56\% & 7\% & 67 & 1250 & 0.010 & 1 & 1823\tabularnewline
					P51 & 10 & 100 & 18257.2 & 24\% & 13\% & 55\% & 8\% & 85 & 4609 & 0.010 & 1 & 1823\tabularnewline
					P52 & 10 & 100 & 23630.2 & 19\% & 12\% & 63\% & 6\% & 73 & 1266 & 0.010 & 1 & 2241\tabularnewline
					P53 & 20 & 100 & 23365.8 & 35\% & 16\% & 41\% & 8\% & 321 & 7000 & 0.010 & 2 & 1137\tabularnewline
					P54 & 10 & 100 & 20205.9 & 23\% & 15\% & 57\% & 6\% & 19 & 1094 & 0.010 & 1 & 2296\tabularnewline
					P55 & 20 & 100 & 15748.3 & 29\% & 19\% & 44\% & 8\% & 53 & 4266 & 0.010 & 1 & 2296\tabularnewline
					\hline 
				\end{tabular}
			}
			\label{Square}
		\end{table}
	\end{landscape}
}

\afterpage{
	\begin{landscape}
		\begin{table}
			\centering
			\caption{Computational results when facility opening cost is a fractional function of service capacity}
			\resizebox{\linewidth}{!}{
				\begin{tabular}{lllllllllllll}
					\hline 
					Instance & \# Service Facilities & \# Customers & Total Cost & Opening Cost & Serving Cost & Accessing Cost & Waiting Cost & \# Iterations & CPU time(ms) & Error Tolerance & \# Open Facilities & Average Service Capacity\tabularnewline
					\hline 
					P1 & 10 & 50 & 26011.7 & 28\% & 19\% & 49\% & 4\% & 10000 & 5032 & 0.027 & 1 & 5746\tabularnewline
					P2 & 10 & 50 & 23711.6 & 31\% & 10\% & 54\% & 5\% & 275 & 172 & 0.010 & 1 & 5746\tabularnewline
					P3 & 10 & 50 & 25711.6 & 36\% & 9\% & 50\% & 5\% & 172 & 230 & 0.010 & 1 & 5746\tabularnewline
					P4 & 10 & 50 & 27711.6 & 41\% & 9\% & 46\% & 4\% & 210 & 156 & 0.010 & 1 & 5746\tabularnewline
					P13 & 20 & 50 & 25723.7 & 29\% & 18\% & 45\% & 7\% & 10000 & 9297 & 0.013 & 1 & 7676\tabularnewline
					P14 & 20 & 50 & 24532.3 & 30\% & 19\% & 44\% & 7\% & 10000 & 9390 & 0.031 & 2 & 4761\tabularnewline
					P15 & 20 & 50 & 26532.3 & 35\% & 18\% & 40\% & 6\% & 10000 & 9641 & 0.029 & 2 & 4761\tabularnewline
					P16 & 20 & 50 & 28532.3 & 40\% & 17\% & 38\% & 6\% & 10000 & 10297 & 0.026 & 2 & 4761\tabularnewline
					P25 & 30 & 150 & 43399.6 & 40\% & 29\% & 28\% & 3\% & 154 & 5968 & 0.010 & 2 & 2965\tabularnewline
					P26 & 30 & 150 & 40419.6 & 35\% & 31\% & 30\% & 3\% & 156 & 5875 & 0.010 & 2 & 2965\tabularnewline
					P27 & 30 & 150 & 44419.6 & 41\% & 29\% & 28\% & 2\% & 155 & 5938 & 0.010 & 2 & 2965\tabularnewline
					P28 & 30 & 150 & 48419.6 & 46\% & 26\% & 25\% & 2\% & 163 & 6032 & 0.010 & 2 & 2965\tabularnewline
					P41 & 10 & 90 & 17397.2 & 42\% & 12\% & 41\% & 4\% & 235 & 406 & 0.010 & 1 & 4894\tabularnewline
					P42 & 20 & 80 & 16369.2 & 45\% & 12\% & 39\% & 4\% & 260 & 1047 & 0.010 & 1 & 4630\tabularnewline
					P43 & 30 & 70 & 14972.5 & 49\% & 12\% & 35\% & 5\% & 278 & 1547 & 0.010 & 1 & 4436\tabularnewline
					P44 & 10 & 90 & 20529.8 & 36\% & 13\% & 48\% & 4\% & 218 & 359 & 0.010 & 1 & 5377\tabularnewline
					P45 & 20 & 80 & 18541.9 & 39\% & 13\% & 44\% & 4\% & 353 & 1157 & 0.010 & 1 & 5082\tabularnewline
					P46 & 30 & 70 & 17500.7 & 42\% & 12\% & 42\% & 4\% & 256 & 1500 & 0.010 & 1 & 4792\tabularnewline
					P47 & 10 & 90 & 19178.4 & 38\% & 14\% & 45\% & 4\% & 93 & 265 & 0.010 & 2 & 3005\tabularnewline
					P48 & 20 & 80 & 17820.9 & 41\% & 13\% & 42\% & 4\% & 120 & 812 & 0.010 & 1 & 5133\tabularnewline
					P49 & 30 & 70 & 17445.8 & 42\% & 12\% & 43\% & 4\% & 181 & 1359 & 0.010 & 2 & 4006\tabularnewline
					P50 & 10 & 100 & 21367.9 & 34\% & 11\% & 51\% & 4\% & 224 & 469 & 0.010 & 1 & 5185\tabularnewline
					P51 & 10 & 100 & 20851.0 & 35\% & 12\% & 49\% & 4\% & 248 & 1531 & 0.010 & 1 & 5185\tabularnewline
					P52 & 10 & 100 & 26165.4 & 28\% & 11\% & 58\% & 3\% & 234 & 469 & 0.010 & 1 & 5638\tabularnewline
					P53 & 20 & 100 & 26193.1 & 28\% & 11\% & 58\% & 3\% & 10000 & 23641 & 0.013 & 2 & 2943\tabularnewline
					P54 & 10 & 100 & 22734.9 & 32\% & 13\% & 51\% & 3\% & 96 & 328 & 0.010 & 1 & 5776\tabularnewline
					P55 & 20 & 100 & 18299.5 & 40\% & 16\% & 39\% & 4\% & 122 & 1203 & 0.010 & 1 & 5776\tabularnewline
					\hline 
				\end{tabular}
			}
			\label{Fractional}
		\end{table}
	\end{landscape}
}	
	
	
	\newpage
	\bibliographystyle{apa}
	\bibliography{mybibfile}
	
\end{document}